\documentclass[12pt]{article}
\usepackage{amssymb,amsmath,amsfonts,amssymb}
\usepackage{amssymb,mathrsfs,bm}
\usepackage[round,comma]{natbib}
\usepackage{graphicx}
\usepackage[all]{xy}
\usepackage{color}

\newtheorem{theorem}{Theorem}

\newtheorem{corollary}[theorem]{Corollary}

\newtheorem{definition}[theorem]{Definition}

\newtheorem{lemma}[theorem]{Lemma}

\numberwithin{equation}{section} \numberwithin{theorem}{section}
\newenvironment{proof}[1][Proof]{\noindent\textbf{#1.} }{\ \rule{0.5em}{0.5em}}

\begin{document}

\vspace{2cm}

{\textbf{\Large{This article has been accepted
for publication \\ on November 2014}}}

\vspace{3cm}

\textbf{To cite this article}:
Subhash C. Kochar and Nuria Torrado,
On stochastic comparisons of largest order statistics
in the scale model,
to appear in \emph{Communications in Statistics - Theory and Methods}.

\vspace{3cm}

Thanks


\newpage

\title{On stochastic comparisons of largest order statistics\\
in the scale model}
\author{ Subhash C. Kochar \\
\small{Fariborz Maseeh Department of Mathematics and Statistics}\\
\small{Portland State University, Portland, OR 97006, USA}\\
Nuria Torrado \\
\small{Centre for Mathematics, University of Coimbra}\\
\small{Apartado 3008, EC Santa Cruz, 3001-501 Coimbra, Portugal} \\
}
\maketitle

\begin{abstract}
Let $X_{\lambda _{1}},X_{\lambda _{2}},\ldots ,X_{\lambda _{n}}$ be
independent nonnegative random variables with $X_{\lambda _{i}}\sim
F(\lambda _{i}t)$, $i=1,\ldots ,n$, where $\lambda _{i}>0$, $i=1,\ldots ,n$
and $F$ is an absolutely continuous distribution. It is shown that, under
some conditions, one largest order statistic $X_{n:n}^{\lambda }$ is smaller
than another one $X_{n:n}^{\theta }$ according to likelihood ratio ordering.
Furthermore, we apply these results when $F$ is a generalized gamma
distribution which includes Weibull, gamma and exponential random variables
as special cases.\newline


\end{abstract}

\bigskip


\section{Introduction}

Let $F$ be the distribution function  of some nonnegative random variable $X$%
. Then the independent random variables $X_{\lambda _{1}},X_{\lambda
_{2}},\ldots ,X_{\lambda _{n}}$ follow the scale model if there exists $%
\lambda _{1}>0,\ldots ,\lambda _{n}>0$ such that, $F_{i}(t)=F(\lambda _{i}t)$
for $i=1,\ldots ,n$. F is called the baseline distribution and the $\lambda
_{i}^{\prime }$s are the scale parameters. Recently, \cite{Khaledi:2011}
studied conditions under which series and parallel systems consisting of
components with lifetimes from the scale family of distributions are ordered
in the hazard rate and the reverse hazard rate orderings, respectively. In
this paper we revisit this problem and broaden the scope of their results to
likelihood ratio ordering, which is stronger than the other orderings.

There is an extensive literature on stochastic orderings among order
statistics and spacings when the observations follow the exponential
distribution with different scale parameters, see for instance, \cite%
{Kochar:1996, Dykstra:1997, Bon:1999, Khaledi:2000, Kochar:2009, Joo:2010,
Torrado:2010, Torrado:2013} and the references therein. Also see a review
paper by \cite{Kochar:2012} on this topic. A natural way to extend these
works is to consider the scale model since it includes the exponential
distribution, among others. The scale model, also known in the literature as
the proportional random variables (PRV) model, is of theoretical as well as
practical importance in various fields of probability and statistics and
has been investigated in \cite{Pledger:1971}, \cite{Hu:1995} and \cite%
{Torrado:2012}, among others.

In this article, we focus on stochastic orders to compare the magnitudes of
two largest order statistics from the scale model when one set of scale
parameters majorizes the other. The new results obtained here are applied
when the baseline distributions are generalized gamma distributions. Recall
that a random variable $X$ has a generalized gamma distribution, denoted by $%
X\sim GG(\beta ,\alpha )$, when its density function has the following form%
\begin{equation*}
f(t)=\frac{\beta }{\Gamma (\frac{\alpha }{\beta })}x^{\alpha -1}e^{-x^{\beta
}},t>0,
\end{equation*}%
where $\beta ,\alpha >0$ are the shapes parameters.
The importance of this distribution lies in its flexibility in describing lifetime distributions
ensuring their applications in
survival analysis and reliability theory.
It is of great interest
in several areas of application, see for example, \cite{Manning:2005, Ali:2008} and \cite{Chen:2012}.
It is well known that generalized gamma distribution includes
many important distributions
like exponential, Weibull  and gamma as special cases.
We also present  some  new results which strengthen  some of those established
earlier in the literature by \cite{Zhao:2011}, \cite{Misra:2013} and \cite%
{ZhaoB:2013} for gamma distributions
and \cite{Torrado:2015} for Weibull distributions.
Further results on these subjects
are contained in,
e.g., \cite{Lihong:2005, Khaledi:2007, ZhaoB:2011, Balakrishnan:2013,
Fang:2013}. It may be mentioned that \cite{Gupta:2006} considered
monotonicity of the hazard rate and the reverse hazard rates of series and
parallel systems when the components are dependent.

The rest of the paper is organized as follows. In Section \ref{sec:def}, we
introduce the required definitions. Sections \ref{sec:parallel} and \ref%
{sec:like} are devoted to investigate the reverse hazard rate and likelihood
ratio orderings of largest order statistics considering the general scale
model, respectively.


\section{Basic definitions}

\label{sec:def}

In this section, we review some definitions and well-known notions of
majorization concepts and stochastic orders. Throughout this article
\emph{increasing} and \emph{non-decreasing} will be
used synonymously as \emph{decreasing} and \emph{%
non-increasing}.

 We focus attention in this article on nonnegative random variables. We shall
also be using the concept of majorization in our discussion. Let $\{
x_{(1)}, x_{(2)},\ldots, x_{(n)} \}$ denote the increasing arrangement of
the components of the vector ${\bm x}=\left(x_{1},x_{2},\ldots,x_{n}\right)$.

\begin{definition}\label{def:01}
The vector ${\bm x}$ is said to be majorized by the vector ${\bm y}$,
denoted by ${\bm x}\overset{m}{\leq}{\bm y}$, if
\begin{equation*}
\sum_{i=1}^{j}x_{(i)}\geq\sum_{i=1}^{j}y_{(i)},\quad\text{for }
j=1,\ldots,n-1 \quad \text{ and } \quad
\sum_{i=1}^{n}x_{(i)}=\sum_{i=1}^{n}y_{(i)}.
\end{equation*}
\end{definition}

Functions that preserve the ordering of majorization are said to be
Schur-convex, as one can see in the following definition.

\begin{definition}
\label{ch0:def:04} A real valued function $\varphi$ defined on a set $%
\mathcal{A}\in\Re^{n}$ is said to be \emph{Schur-convex} (\emph{Schur-concave%
}) on $\mathcal{A}$ if
\begin{equation*}
{\bm x} \overset{m}{\leq}{\bm y}\, \text{on }\,\mathcal{A}\Rightarrow \varphi({\bm x}
)\leq (\geq) \varphi({\bm y} ).
\end{equation*}
\end{definition}


Replacing  the equality in Definition \ref{def:01} by a corresponding inequality
leads to the concept of weak majorization. One can majorize from above
or below. The following definition addresses majorization from above. It is  also called \emph{supermajorization}.

\begin{definition}
The vector ${\bm x}$ is said to be weakly majorized by the vector ${\bm y}$,
denoted by ${\bm x}\overset{w}{\leq}{\bm y}$, if
\begin{equation*}
\sum_{i=1}^{j}x_{(i)}\geq\sum_{i=1}^{j}y_{(i)},\quad\text{for } j=1,\ldots,n.
\end{equation*}
\end{definition}

It is known that ${\bm x}\overset{m}{\leq}{\bm y}\Rightarrow {\bm x}\overset{w}{\leq}{\bm y}$.
The converse is, however, not true. For extensive and comprehensive details
on the theory of majorization orders and their applications, please refer to
the book of \cite{Marshall:2011}.

Let $X$ and $Y$ be univariate random variables with cumulative distribution
functions (c.d.f.'s) $F$ and $G$, survival functions $\Bar{F}\left(= 1 -
F\right)$ and $\Bar{G}\left(= 1 - G\right)$, p.d.f.'s $f$ and $g$, hazard
rate functions $h_{F}\left(=f/\,\Bar{F}\right)$ and $h_{G}\left(=g/\,\Bar{G}%
\right)$, and reverse hazard rate functions $r_{F}\left(=f/F\right)$ and $%
r_{G}\left(=g/G\right)$, respectively. The following definitions introduce
stochastic orders, which are considered in this article, to compare the
magnitudes of two random variables. For more details on stochastic
comparisons, see \cite{Shaked:2007}.

\begin{definition}
We say that $X$ is smaller than $Y$ in the:

\begin{itemize}
\item[a)] usual stochastic order if $\Bar{F}(t)\leq\Bar{G}(t)$ for all $t$
and in this case, we write $X\leq_{st}Y$,

\item[b)] reverse hazard rate order if $G(t)/F(t)$ is increasing in $t$ for
which the ratio is well defined, or if $r_{F}(t)\leq r_{G}(t)$, for all $t$,
denoted by $X\leq_{rh}Y$,

\item[c)] likelihood ratio order if $g(t)/f(t)$ is increasing in $t$ for
which the ratio is well defined, for all $t$, denoted by $X\leq_{lr}Y$.
\end{itemize}
\end{definition}



\section{Reverse hazard rate ordering results}

\label{sec:parallel}

Let $X_{\lambda _{1}},X_{\lambda _{2}},\ldots ,X_{\lambda _{n}}$ be
independent nonnegative random variables with $X_{\lambda _{i}}\sim
F(\lambda _{i}t)$, $i=1,\ldots ,n$, where $\lambda _{i}>0,i=1,\ldots ,n$ and
$F$ is an absolutely continuous distribution. Let $f$, $h$ and $r$\ be the
density, hazard rate and the reverse hazard rate functions of $F$,
respectively. The distribution function of $X_{n:n}^{\lambda }$, the largest
order statistic formed from $X_{\lambda _{1}},X_{\lambda _{2}},\ldots
,X_{\lambda _{n}}$ is
\begin{equation*}
F_{n:n}^{\lambda }(t)=\prod\limits_{i=1}^{n}F(\lambda _{i}t),
\end{equation*}%
and its reverse hazard rate function is%
\begin{equation}
r_{n:n}^{\lambda }(t)=\sum_{i=1}^{n}\lambda _{i}r(\lambda _{i}t).
\label{ecrh}
\end{equation}

\cite{Khaledi:2011} proved the following result on comparing two parallel
systems when the underlying random variables follow the scale model and
their scale parameters majorize each other.

\begin{theorem}
\label{rhmto} Let $X_{\lambda _{1}},\ldots ,X_{\lambda _{n}}$ be independent
nonnegative random variables with $X_{\lambda _{i}}\sim F(\lambda _{i}t)$, $%
i=1,\ldots ,n$, where $\lambda _{i}>0$, $i=1,\ldots ,n$ and $F$ is an
absolutely continuous distribution. Let ${r}$ be the reverse hazard rate
function of $F$, respectively. If $t^{2}{r}\,^{\prime }(t)$ is increasing in
$t$, then
\begin{equation*}
\left( \lambda _{1},\ldots ,\lambda _{n}\right) \overset{m}{\leq}\left( \theta
_{1},\ldots ,\theta _{n}\right) \Rightarrow X_{n:n}^{\lambda }\leq
_{rh}X_{n:n}^{\theta }.
\end{equation*}
\end{theorem}

In the next theorem we extend the above result to the case when the two sets
of scale parameters weakly majorize each other instead of usual majorization.

\begin{theorem}
\label{th02}Let $X_{\lambda _{1}},X_{\lambda _{2}},\ldots ,X_{\lambda _{n}}$
be independent random variables with $X_{\lambda _{i}}\sim F(\lambda _{i}t)$
where $\lambda _{i}>0$, $i=1,\ldots ,n$. If $tr(t)$ is decreasing in $t$ and
$t^{2}r^{\prime }(t)$ is increasing in $t$, then
\begin{equation*}
\left( \lambda _{1},\ldots ,\lambda _{n}\right) \overset{w}{\leq}\left( \theta
_{1},\ldots ,\theta _{n}\right) \Rightarrow X_{n:n}^{\lambda }\leq
_{rh}X_{n:n}^{\theta }.
\end{equation*}
\end{theorem}

\begin{proof}
Fix $t>0$. Then the reverse hazard rate of $X_{n:n}^{\lambda }$ as given by
(\ref{ecrh}) can be rewritten as%
\begin{equation*}
r_{n:n}^{\lambda }(t)=\sum_{i=1}^{n}\lambda _{i}r(\lambda _{i}t)=\frac{1}{t}%
\sum_{i=1}^{n}\psi (\lambda _{i}t),
\end{equation*}%
where $\psi (t)=tr(t)$, $t\geq 0$.
From Theorem A.8 of \cite{Marshall:2011} (p. 59) it suffices to show that,
for each $t>0$, $r_{n:n}^{\lambda }(t)$ is decreasing in each $\lambda _{i}$%
, $i=1,\ldots ,n$, and is a Schur-convex function of $\left( \lambda
_{1},\ldots ,\lambda _{n}\right) $. By the assumptions, $tr(t)$ is
decreasing in $t$, then the reverse hazard rate function of $X_{n:n}$ is
decreasing in each $\lambda _{i}$.

Now, from Proposition C.1 of \cite{Marshall:2011} (p. 64), the convexity of $%
\psi (t)$ is needed to prove Schur-convexity of $r_{n:n}^{\lambda }(t)$.
Note that the assumption $t^{2}r^{\prime }(t)$ is increasing in $t$ is
equivalent to $r(t)+tr^{\prime }(t)$ is increasing in $t$ since%
\begin{equation*}
\left[ t^{2}r^{\prime }(t)\right] ^{\prime }=t\left( 2r^{\prime
}(t)+tr^{\prime \prime }(t)\right) =t\left[ r(t)+tr^{\prime }(t)\right]
^{\prime },
\end{equation*}%
and $r(t)+tr^{\prime }(t)$ is increasing in $t$ is equivalent to $tr(t)$ is
convex since%
\begin{equation*}
\left[ tr(t)\right] ^{\prime }=r(t)+tr^{\prime }(t)\text{.}
\end{equation*}%
Hence, $\psi (t)$ is convex.
\end{proof}

Note that the conditions of Theorem \ref{th02} are satisfied by the
generalized gamma distribution with parameters $\beta \leq 1$ and $\alpha >0$
as \cite{Khaledi:2011} proved that $t^{2}r^{\prime }(t)$ is an increasing
function for $X\sim GG(\beta ,\alpha )$, when $\beta \leq 1$ and $\alpha >0$%
. It is easy to verify that $tr(t)$ is a decreasing function of $t$ when $%
\beta ,\alpha>0$.

As one natural application, Theorem \ref{th02} guarantees that, for parallel
systems of components having independent generalized gamma distributed lifetimes with
parameters $\beta \leq 1$ and $\alpha >0$, the weakly majorized scale parameter vector leads to a larger system's lifetime in the sense of the reverse hazard rate order.

The generalized gamma distribution includes many important distributions
like exponential ($\beta =\alpha =1$), Weibull ($\beta =\alpha $) and gamma (%
$\beta =1$) as special cases. \cite{Misra:2013} proved that in the
case of gamma distribution with density function%
\begin{equation*}
f(t)=\frac{\lambda _{i}^{\alpha }}{\Gamma (\alpha )}t^{\alpha -1}e^{-\lambda
_{i}t},\text{ }t>0,
\end{equation*}%
when $\alpha >0$ and $n\geq 2$,
\begin{equation}
\left( \lambda _{1},\ldots ,\lambda _{n}\right) \overset{w}{\leq}\left( \theta
_{1},\ldots ,\theta _{n}\right) \Rightarrow X_{n:n}^{\lambda }\leq
_{rh}X_{n:n}^{\theta }.  \label{ec:misra}
\end{equation}%
Note that, when $0<\alpha \leq 1$, (\ref{ec:misra}) can be seen as a
particular case of Theorem \ref{th02} since gamma distribution is a
particular case of generalized gamma distribution when $\beta =1$.

Recently, \cite{Torrado:2015} established, in Theorem 4.1, the
reverse hazard rate ordering between parallel systems based on two sets of heterogeneous
Weibull random variables with a common shape parameter and with scale parameters which
are ordered according to a majorization order
when the common shape parameter $\alpha$ satisfies $0 < \alpha \leq 1$.
So Theorem 4.1 in \cite{Torrado:2015} can be seen as a
particular case of Theorem \ref{th02} because
Weibull distribution is a
particular case of generalized gamma distribution when $\beta =\alpha$.

\bigskip


\section{Likelihood ratio ordering results}

\label{sec:like}

In this section, we investigate whether the result of Theorem \ref{th02} can
be strengthened from reverse hazard rate ordering to likelihood ratio
ordering. First, we consider the case when $n=2$ and the scale parameters of
the scale model are ordered according to a weekly majorization order.

\begin{theorem}
\label{th01}Let $X_{\lambda _{1}},X_{\lambda }$ be independent nonnegative
random variables with $X_{\lambda _{1}}\sim F(\lambda _{1}t)$ and $%
X_{\lambda }\sim F(\lambda t)$, where $\lambda _{1},\lambda >0$ and $F$ is
an absolutely continuous distribution. Let $Y_{\lambda _{1}^{\ast
}},Y_{\lambda }$ be independent nonnegative random variables with $%
Y_{\lambda _{1}^{\ast }}\sim F(\lambda _{1}^{\ast }t)$ and $Y_{\lambda }\sim
F(\lambda t)$, where $\lambda _{1}^{\ast },\lambda >0$. Let $r$ be the
reverse hazard rate function of $F$. Assume $tr(t)$ and $tr^{\prime }(t)/r(t)
$ are both decreasing in $t$. Suppose $\lambda _{1}^{\ast }=\min (\lambda
,\lambda _{1},\lambda _{1}^{\ast })$, then%
\begin{equation*}
\left( \lambda _{1},\lambda \right) \overset{w}{\leq}\left( \lambda _{1}^{\ast
},\lambda \right) \Rightarrow \frac{r_{2:2}^{\ast }(t)}{r_{2:2}(t)}\text{ is
increasing in }t.
\end{equation*}
\end{theorem}

\begin{proof}
Let
\begin{equation*}
\phi (t)=\frac{r_{2:2}^{\ast }(t)}{r_{2:2}(t)}=\frac{\lambda _{1}^{\ast
}r(\lambda _{1}^{\ast }t)+\lambda r(\lambda t)}{\lambda _{1}r(\lambda
_{1}t)+\lambda r(\lambda t)}\text{,}
\end{equation*}%
and its derivative for $t > 0$ is,
\begin{eqnarray*}
\phi ^{\prime }(t) &\overset{\text{sign}}{=} &
t^{3}\left( \left( \lambda _{1}^{\ast
}\right) ^{2}r^{\prime }(\lambda _{1}^{\ast }t)+\lambda ^{2}r^{\prime
}(\lambda t)\right) \left( \lambda _{1}r(\lambda _{1}t)+\lambda r(\lambda
t)\right) \\
&&-t^{3}\left( \lambda _{1}^{\ast }r(\lambda _{1}^{\ast }t)+\lambda
r(\lambda t)\right) \left( \lambda _{1}^{2}r^{\prime }(\lambda
_{1}t)+\lambda ^{2}r^{\prime }(\lambda t)\right) \\
&=&\lambda _{1}\lambda _{1}^{\ast }t^{3}\left( \lambda _{1}^{\ast }r^{\prime
}(\lambda _{1}^{\ast }t)r(\lambda _{1}t)-\lambda _{1}r(\lambda _{1}^{\ast
}t)r^{\prime }(\lambda _{1}t)\right) \\
&&+\lambda _{1}\lambda t^{3}\left( \lambda r^{\prime }(\lambda t)r(\lambda
_{1}t)-\lambda _{1}r(\lambda t)r^{\prime }(\lambda _{1}t)\right) \\
&&+\lambda \lambda _{1}^{\ast }t^{3}\left( \lambda _{1}^{\ast }r^{\prime
}(\lambda _{1}^{\ast }t)r(\lambda t)-\lambda r(\lambda _{1}^{\ast
}t)r^{\prime }(\lambda t)\right) \\
&=&\lambda _{1}\lambda _{1}^{\ast }t^{2}r(\lambda _{1}^{\ast }t)r(\lambda
_{1}t)\left( \lambda _{1}^{\ast }t\frac{r^{\prime }(\lambda _{1}^{\ast }t)}{%
r(\lambda _{1}^{\ast }t)}-\lambda _{1}t\frac{r^{\prime }(\lambda _{1}t)}{%
r(\lambda _{1}t)}\right) \\
&&+\lambda _{1}\lambda t^{2}r(\lambda t)r(\lambda _{1}t)\left( \lambda t%
\frac{r^{\prime }(\lambda t)}{r(\lambda t)}-\lambda _{1}t\frac{r^{\prime
}(\lambda _{1}t)}{r(\lambda _{1}t)}\right) \\
&&+\lambda \lambda _{1}^{\ast }t^{2}r(\lambda _{1}^{\ast }t)r(\lambda
t)\left( \lambda _{1}^{\ast }t\frac{r^{\prime }(\lambda _{1}^{\ast }t)}{%
r(\lambda _{1}^{\ast }t)}-\lambda t\frac{r^{\prime }(\lambda t)}{r(\lambda t)%
}\right) \\
&=&\psi (\lambda _{1}^{\ast }t)\psi (\lambda _{1}t)\left( -\eta (\lambda
_{1}^{\ast }t)+\eta (\lambda _{1}t)\right) +\psi (\lambda t)\psi (\lambda
_{1}t)\left( -\eta (\lambda t)+\eta (\lambda _{1}t)\right) \\
&&+\psi (\lambda _{1}^{\ast }t)\psi (\lambda t)\left( -\eta (\lambda
_{1}^{\ast }t)+\eta (\lambda t)\right) \text{,}
\end{eqnarray*}%
where
\begin{equation*}
\psi \left( t\right) =tr(t)\text{ and }\eta \left( t\right) =-t\frac{%
r^{\prime }(t)}{r(t)}.
\end{equation*}%
Note that $\psi \left( t\right) \geq 0$ for all $t\geq 0$ and $\eta \left(
t\right) \geq 0$ since $r^{\prime }(t)\leq 0$ because $tr(t)$ is a
decreasing function. By the assumptions, we know that $\psi \left( t\right) $
is decreasing and $\eta \left( t\right) $ is increasing in $t$. If $\lambda
_{1}^{\ast }=\min (\lambda ,\lambda _{1},\lambda _{1}^{\ast })$ and $\left(
\lambda _{1},\lambda \right) \overset{w}{\leq}
\left( \lambda _{1}^{\ast },\lambda
\right) $, then $\lambda _{1}^{\ast }\leq \lambda \leq \lambda _{1}$ or $%
\lambda _{1}^{\ast }\leq \lambda _{1}\leq \lambda $. When $\lambda
_{1}^{\ast }\leq \lambda \leq \lambda _{1}$, we have%
\begin{eqnarray*}
\phi ^{\prime }(t) &\overset{\text{sign}}{=} &
\psi (\lambda _{1}^{\ast }t)\psi
(\lambda _{1}t)\left( -\eta (\lambda _{1}^{\ast }t)+\eta (\lambda
_{1}t)\right) +\psi (\lambda t)\psi (\lambda _{1}t)\left( -\eta (\lambda
t)+\eta (\lambda _{1}t)\right) \\
&&+\psi (\lambda _{1}^{\ast }t)\psi (\lambda t)\left( -\eta (\lambda
_{1}^{\ast }t)+\eta (\lambda t)\right) \\
&\geq &0\text{,}
\end{eqnarray*}%
since $\eta \left( \lambda _{1}^{\ast }t\right) \leq \eta \left( \lambda
t\right) \leq \eta \left( \lambda _{1}t\right) $. When $\lambda _{1}^{\ast
}\leq \lambda _{1}\leq \lambda $, we get
\begin{eqnarray*}
\phi ^{\prime }(t) &\geq &\psi (\lambda t)\psi (\lambda _{1}t)\left( -\eta
(\lambda _{1}^{\ast }t)+\eta (\lambda _{1}t)\right) +\psi (\lambda t)\psi
(\lambda _{1}t)\left( -\eta (\lambda t)+\eta (\lambda _{1}t)\right) \\
&&+\psi (\lambda _{1}t)\psi (\lambda t)\left( -\eta (\lambda _{1}^{\ast
}t)+\eta (\lambda t)\right) \\
&=&2\psi (\lambda t)\psi (\lambda _{1}t)\left( -\eta (\lambda _{1}^{\ast
}t)+\eta (\lambda _{1}t)\right) \geq 0\text{.}
\end{eqnarray*}%
Therefore $r_{2:2}^{\ast }(t)/r_{2:2}(t)$ is increasing in $t$.
\end{proof}


In the next result, we extend Theorem \ref{th02} from reverse hazard rate ordering to likelihood ratio
ordering for $n=2$.


\begin{theorem}
\label{th05}Let $X_{\lambda _{1}},X_{\lambda }$ be independent nonnegative
random variables with $X_{\lambda _{1}}\sim F(\lambda _{1}t)$ and $%
X_{\lambda }\sim F(\lambda t)$, where $\lambda _{1},\lambda >0$ and $F$ is
an absolutely continuous distribution. Let $r$ be the reverse hazard rate
function of $F$. Let $Y_{\lambda _{1}^{\ast }},Y_{\lambda }$ be independent
nonnegative random variables with $Y_{\lambda _{1}^{\ast }}\sim F(\lambda
_{1}^{\ast }t)$ and $Y_{\lambda }\sim F(\lambda t)$, where $\lambda
_{1}^{\ast },\lambda >0$. Assume $tr(t)$ and $tr^{\prime }(t)/r(t)$ are both
decreasing in $t$ and $t^{2}r^{\prime }(t)$ is increasing in $t$. Suppose $%
\lambda _{1}^{\ast }=\min (\lambda ,\lambda _{1},\lambda _{1}^{\ast })$, then%
\begin{equation*}
\left( \lambda _{1},\lambda \right) \overset{w}{\leq}\left( \lambda _{1}^{\ast
},\lambda \right) \Rightarrow X_{2:2}\leq _{lr}Y_{2:2}.
\end{equation*}
\end{theorem}

\begin{proof}
From Theorem \ref{th01}, we know that $r_{2:2}^{\ast }(t)/r_{2:2}(t)$ is
increasing in $t$ under the given assumptions. By Theorem \ref{th02}, $%
\left( \lambda _{1},\lambda \right) \overset{w}{\leq}\left( \lambda _{1}^{\ast
},\lambda \right) $ implies $X_{2:2}\leq _{rh}Y_{2:2}$. Thus the required
result follows from Theorem 1.C.4 of \cite{Shaked:2007}.
\end{proof}

The conditions of Theorem \ref{th05} hold when the baseline distribution in
the scale model is $GG(\beta ,\alpha )$ with parameters $\alpha \leq \beta
\leq 1$. We know from \cite{Khaledi:2011} that for $\alpha ,\beta >0$, the
function $tr(t)$ is decreasing in $t$ and for $\beta \leq 1$ and $\alpha >0$%
, the function $t^{2}r^{\prime }(t)$ is increasing in $t$. In Lemma \ref%
{lem01}, we show that the function $tr^{\prime }(t)/r(t)$ is decreasing in $t
$ when $\alpha \leq \beta $.

\begin{lemma}
\label{lem01} Let $X\sim GG(\beta ,\alpha )$, $\alpha \leq \beta $, with
reverse hazard rate $r(t)$, then $tr^{\prime }(t)/r(t)$ is a decreasing
function.
\end{lemma}

\begin{proof}
The reverse hazard rate of $GG(\beta ,\alpha )$ is
\begin{equation*}
r(t)=\frac{t^{\alpha -1}e^{-t^{\beta }}}{\int_{0}^{t}x^{\alpha
-1}e^{-x^{\beta }}dx}\text{.}
\end{equation*}%
From (A.21) in \cite{Khaledi:2011}, we know%
\begin{equation}
t\frac{r^{\prime }(t)}{r(t)}=\alpha -1-\beta t^{\beta }-tr(t)\text{.}
\label{ec01}
\end{equation}%
Differentiating with respect to $t$, we get%
\begin{equation*}
\left[ t\frac{r^{\prime }(t)}{r(t)}\right] ^{\prime }=-\beta ^{2}t^{\beta
-1}-r(t)-tr^{\prime }(t)\text{.}
\end{equation*}%
Note that, in general, the derivative of any reverse hazard rate with
respect to $t$ is%
\begin{equation}
r^{\prime }(t)=\frac{f^{\prime }(t)}{F(t)}-r^{2}(t)\text{.}  \label{ec02}
\end{equation}%
Combining these observations, we have
\begin{eqnarray*}
\left[ t\frac{r^{\prime }(t)}{r(t)}\right] ^{\prime } &=&-\beta ^{2}t^{\beta
-1}-r(t)-t\frac{f^{\prime }(t)}{F(t)}+tr^{2}(t) \\
&=&-\beta ^{2}t^{\beta -1}+r(t)\left( tr(t)-1-t\frac{f^{\prime }(t)}{f(t)}%
\right) \text{.}
\end{eqnarray*}%
From \cite{Khaledi:2011}, we know that $tr(t)$ is a decreasing function for $%
\beta ,\alpha >0$ and also that $\lim_{t\rightarrow 0}tr(t)=\alpha $ and $%
\lim_{t\rightarrow \infty }tr(t)=0$, then $tr(t)\leq \alpha $ for all $t>0$.
Then%
\begin{equation}
\left[ t\frac{r^{\prime }(t)}{r(t)}\right] ^{\prime }\leq -\beta
^{2}t^{\beta -1}+r(t)\left( \alpha -1-t\frac{f^{\prime }(t)}{f(t)}\right)
\text{.}  \label{ec03}
\end{equation}%
From (\ref{ec01}) and (\ref{ec02}), we get%
\begin{equation*}
t\frac{r^{\prime }(t)}{r(t)}=\frac{t}{r(t)}\left( \frac{f^{\prime }(t)}{F(t)}%
-r^{2}(t)\right) =t\left( \frac{f^{\prime }(t)}{f(t)}-r(t)\right) \text{,}
\end{equation*}%
then%
\begin{eqnarray*}
t\frac{f^{\prime }(t)}{f(t)} &=&\alpha -1-\beta t^{\beta }-tr(t)+tr(t) \\
&=&\alpha -1-\beta t^{\beta }\text{.}
\end{eqnarray*}%
By replacing the above expression in (\ref{ec03}), we have%
\begin{eqnarray*}
\left[ t\frac{r^{\prime }(t)}{r(t)}\right] ^{\prime } &\leq &-\beta
^{2}t^{\beta -1}+r(t)\left( \alpha -1-\left( \alpha -1-\beta t^{\beta
}\right) \right)  \\
&=&-\beta ^{2}t^{\beta -1}+\beta t^{\beta }r(t) \\
&=&\beta t^{\beta }\left( -\frac{\beta }{t}+r(t)\right) \leq 0
\end{eqnarray*}%
since $tr(t)\leq \alpha \leq \beta $.
\end{proof}


Theorem \ref{th05} says that the lifetime of a parallel system consisting of two types of generalized gamma components
with parameters $\alpha \leq \beta
\leq 1$ is stochastically larger according to likelihood ratio
ordering when the scale parameters are more dispersed according to weakly majorization.

\bigskip

Note that, when $0<\alpha \leq 1$, Theorem 3.4 in \cite{Zhao:2011} can be
seen as a particular case of Theorem \ref{th05} since gamma distribution is
a particular case of generalized gamma distribution when $\beta =1$.

As an immediate consequence of Theorem \ref{th05}, we have the following
result which provides an upper bound of two random variables from a scale
model.

\begin{corollary}
Let $X_{\lambda _{1}},X_{\lambda _{2}}$ be independent nonnegative random
variables with $X_{\lambda _{i}}\sim F(\lambda _{i}t)$ for $i=1,2$. Let $%
Y_{1},Y_{2}$ be independent nonnegative random variables with a common
distribution $Y_{i}\sim F(\lambda t)$ for $i=1,2$. Assume $tr(t)$ and $%
tr^{\prime }(t)/r(t)$ are both decreasing in $t$ and $t^{2}r^{\prime }(t)$
is increasing in $t$. Suppose $\lambda \leq \min (\lambda _{1},\lambda _{2})$%
, then%
\begin{equation*}
\lambda \leq \frac{\lambda _{1}+\lambda _{2}}{2}\Rightarrow X_{2:2}\leq
_{lr}Y_{2:2}.
\end{equation*}
\end{corollary}


Next, we extend the study of likelihood ratio ordering between largest order
statistics from the two-variable case to multiple-outlier scale models.

\begin{theorem}
\label{th12}Let $X_{1},\ldots ,X_{n}$ be independent nonnegative random
variables such that $X_{i}\sim F(\lambda _{1}t)$ for $i=1,\ldots ,p$\ and $%
X_{j}\sim F(\lambda t)$ for $j=p+1,\ldots ,n$, with $\lambda _{1},\lambda >0$
and $F$ is an absolutely continuous distribution. Let $Y_{1},\ldots ,Y_{n}$
$n$
be independent nonnegative random variables with $Y_{i}\sim F(\lambda
_{1}^{\ast }t)$ for $i=1,\ldots ,p$ and $Y_{j}\sim F(\lambda t)$ for $%
j=p+1,\ldots ,n$, with $\lambda _{1}^{\ast },\lambda >0$. Let $r$ be the
reverse hazard rate function of $F$. Assume $tr(t)$ and $tr^{\prime }(t)/r(t)
$ are both decreasing in $t$. Suppose $\lambda _{1}^{\ast }=\min (\lambda
,\lambda _{1},\lambda _{1}^{\ast })$, then%
\begin{equation*}
(\underbrace{\lambda _{1},\ldots ,\lambda _{1}}_{p},\underbrace{%
\lambda,\ldots ,\lambda}_{q}) \overset{w}{\leq} (\underbrace{\lambda _{1}^{\ast
},\ldots ,\lambda _{1}^{\ast }}_{p},\underbrace{\lambda,\ldots ,\lambda}%
_{q}) \Rightarrow \frac{r_{n:n}^{\ast }(t)}{r_{n:n}(t)}\text{ is increasing
in }t,
\end{equation*}
where $q=n-p$.

\end{theorem}

\begin{proof}
From (\ref{ecrh}) we get the reverse hazard rate function of $X_{n:n}$:%
\begin{equation*}
r_{n:n}(t)=p\lambda _{1}r(\lambda _{1}t)+q\lambda r(\lambda t),
\end{equation*}%
where $p+q=n$. Let%
\begin{equation*}
\phi (t)=\frac{r_{n:n}^{\ast }(t)}{r_{n:n}(t)}=\frac{p\lambda _{1}^{\ast
}r(\lambda _{1}^{\ast }t)+q\lambda r(\lambda t)}{p\lambda _{1}r(\lambda
_{1}t)+q\lambda r(\lambda t)}\text{.}
\end{equation*}%
On differentiating $\phi (t)$ with respect to $t$, we get%
\begin{eqnarray*}
\phi ^{\prime }(t) &\overset{\text{sign}}{=}&
t^{3}\left( p\left( \lambda _{1}^{\ast
}\right) ^{2}r^{\prime }(\lambda _{1}^{\ast }t)+q\lambda ^{2}r^{\prime
}(\lambda t)\right) \left( p\lambda _{1}r(\lambda _{1}t)+q\lambda r(\lambda
t)\right) \\
&&-t^{3}\left( p\lambda _{1}^{\ast }r(\lambda _{1}^{\ast }t)+q\lambda
r(\lambda t)\right) \left( p\lambda _{1}^{2}r^{\prime }(\lambda
_{1}t)+q\lambda ^{2}r^{\prime }(\lambda t)\right) \\
&=&p^{2}\lambda _{1}\lambda _{1}^{\ast }t^{3}\left( \lambda _{1}^{\ast
}r^{\prime }(\lambda _{1}^{\ast }t)r(\lambda _{1}t)-\lambda _{1}r(\lambda
_{1}^{\ast }t)r^{\prime }(\lambda _{1}t)\right) \\
&&+pq\lambda _{1}\lambda t^{3}\left( \lambda r^{\prime }(\lambda t)r(\lambda
_{1}t)-\lambda _{1}r(\lambda t)r^{\prime }(\lambda _{1}t)\right) \\
&&+pq\lambda \lambda _{1}^{\ast }t^{3}\left( \lambda _{1}^{\ast }r^{\prime
}(\lambda _{1}^{\ast }t)r(\lambda t)-\lambda r(\lambda _{1}^{\ast
}t)r^{\prime }(\lambda t)\right) \\
&=&p^{2}\lambda _{1}\lambda _{1}^{\ast }t^{2}r(\lambda _{1}^{\ast
}t)r(\lambda _{1}t)\left( \lambda _{1}^{\ast }t\frac{r^{\prime }(\lambda
_{1}^{\ast }t)}{r(\lambda _{1}^{\ast }t)}-\lambda _{1}t\frac{r^{\prime
}(\lambda _{1}t)}{r(\lambda _{1}t)}\right) \\
&&+pq\lambda _{1}\lambda t^{2}r(\lambda t)r(\lambda _{1}t)\left( \lambda t%
\frac{r^{\prime }(\lambda t)}{r(\lambda t)}-\lambda _{1}t\frac{r^{\prime
}(\lambda _{1}t)}{r(\lambda _{1}t)}\right) \\
&&+pq\lambda \lambda _{1}^{\ast }t^{2}r(\lambda _{1}^{\ast }t)r(\lambda
t)\left( \lambda _{1}^{\ast }t\frac{r^{\prime }(\lambda _{1}^{\ast }t)}{%
r(\lambda _{1}^{\ast }t)}-\lambda t\frac{r^{\prime }(\lambda t)}{r(\lambda t)%
}\right) \\
&=&p^{2}\psi (\lambda _{1}^{\ast }t)\psi (\lambda _{1}t)\left( -\eta
(\lambda _{1}^{\ast }t)+\eta (\lambda _{1}t)\right) +pq\psi (\lambda t)\psi
(\lambda _{1}t)\left( -\eta (\lambda t)+\eta (\lambda _{1}t)\right) \\
&&+pq\psi (\lambda _{1}^{\ast }t)\psi (\lambda t)\left( -\eta (\lambda
_{1}^{\ast }t)+\eta (\lambda t)\right) \text{,}
\end{eqnarray*}%
where
\begin{equation*}
\psi \left( t\right) =tr(t)\text{ and }\eta \left( t\right) =-t\frac{%
r^{\prime }(t)}{r(t)}.
\end{equation*}%
Note that $\psi \left( t\right) ,\eta \left( t\right) \geq 0$ for all $t\geq
0$. By the assumptions, we know that $\psi \left( t\right) $ is decreasing
and $\eta \left( t\right) $ is increasing in $t$. If $\lambda _{1}^{\ast
}=\min (\lambda ,\lambda _{1},\lambda _{1}^{\ast })$ and $\left( \lambda
_{1},\ldots ,\lambda _{1},\lambda ,\ldots ,\lambda \right)
 \overset{w}{\leq}
\left( \lambda _{1}^{\ast },\ldots ,\lambda _{1}^{\ast },\lambda ,\ldots
,\lambda \right) $, then $\lambda _{1}^{\ast }\leq \lambda \leq \lambda _{1}$
or $\lambda _{1}^{\ast }\leq \lambda _{1}\leq \lambda $. When $\lambda
_{1}^{\ast }\leq \lambda \leq \lambda _{1}$, it is easy to check that $\phi
^{\prime }(t)\geq 0$ since $\eta \left( \lambda _{1}^{\ast }t\right) \leq
\eta \left( \lambda t\right) \leq \eta \left( \lambda _{1}t\right) $. When $%
\lambda _{1}^{\ast }\leq \lambda _{1}\leq \lambda $, we get
\begin{eqnarray*}
\phi ^{\prime }(t) &\geq &p^{2}\psi (\lambda t)\psi (\lambda _{1}t)\left(
-\eta (\lambda _{1}^{\ast }t)+\eta (\lambda _{1}t)\right) +pq\psi (\lambda
t)\psi (\lambda _{1}t)\left( -\eta (\lambda t)+\eta (\lambda _{1}t)\right) \\
&&+pq\psi (\lambda _{1}t)\psi (\lambda t)\left( -\eta (\lambda _{1}^{\ast
}t)+\eta (\lambda t)\right) \\
&=&np\psi (\lambda t)\psi (\lambda _{1}t)\left( -\eta (\lambda _{1}^{\ast
}t)+\eta (\lambda _{1}t)\right) \geq 0\text{.}
\end{eqnarray*}%
Therefore $r_{n:n}^{\ast }(t)/r_{n:n}(t)$ is increasing in $t$.
\end{proof}


In the next result, we extend Theorem \ref{th05}
from the two-variable case to multiple-outlier scale models.


\begin{theorem}
\label{th13}Let $X_{1},\ldots ,X_{n}$ be independent nonnegative random
variables such that $X_{i}\sim F(\lambda _{1}t)$ for $i=1,\ldots ,p$\ and $%
X_{j}\sim F(\lambda t)$ for $j=p+1,\ldots ,n$, with $\lambda _{1},\lambda >0$
and $F$ is an absolutely continuous distribution. Let $Y_{1},\ldots ,Y_{n}$
$n$
be independent nonnegative random variables with $Y_{i}\sim F(\lambda
_{1}^{\ast }t)$ for $i=1,\ldots ,p$ and $Y_{j}\sim F(\lambda t)$ for $%
j=p+1,\ldots ,n$, with $\lambda _{1}^{\ast },\lambda >0$. Let $r$ be the
reverse hazard rate function of $F$. Assume $tr(t)$ and $tr^{\prime }(t)/r(t)
$ are both decreasing in $t$ and $t^{2}r^{\prime }(t)$ is increasing in $t$.
Suppose $\lambda _{1}^{\ast }=\min (\lambda ,\lambda _{1},\lambda _{1}^{\ast
})$, then%
\begin{equation*}
(\underbrace{\lambda _{1},\ldots ,\lambda _{1}}_{p},\underbrace{%
\lambda,\ldots ,\lambda}_{q})  \overset{w}{\leq} (\underbrace{\lambda _{1}^{\ast
},\ldots ,\lambda _{1}^{\ast }}_{p},\underbrace{\lambda,\ldots ,\lambda}%
_{q}) \Rightarrow X_{n:n}\leq _{lr}Y_{n:n}.
\end{equation*}
\end{theorem}

\begin{proof}
From Theorem \ref{th12}, we know that $r_{n:n}^{\ast }(t)/r_{n:n}(t)$ is
increasing in $t$ when
$tr(t)$ and $tr^{\prime }(t)/r(t)
$ are both decreasing in $t$.
%
Since $\left( \lambda _{1},\ldots
,\lambda _{1},\lambda ,\ldots ,\lambda \right)  \overset{w}{\leq}\left( \lambda
_{1}^{\ast },\ldots ,\lambda _{1}^{\ast },\lambda ,\ldots ,\lambda \right) $%
and $t^{2}r^{\prime }(t)$ is increasing in $t$
, then $X_{n:n}\leq _{rh}Y_{n:n}$ from Theorem \ref{th02}. Thus the required
result follows from Theorem 1.C.4 in \cite{Shaked:2007}.
\end{proof}


Note that, when $0<\alpha \leq 1$, Theorem 3.1 in \cite{ZhaoB:2013} can be
seen as a particular case of Theorem \ref{th13} when $\lambda _{1}^{\ast
}\leq \lambda _{1}\leq \lambda $ since gamma distribution is a particular
case of generalized gamma distribution when $\beta =1$.

Next, we establish the analog of Theorem \ref{th13} when both the baseline
distributions and the scale parameters are different in the multiple-outlier
scale models.

\begin{theorem}
\label{th21}Let $X_{1},\ldots ,X_{n}$ be independent nonnegative random
variables such that $X_{i}\sim F(\lambda _{1}t)$ for $i=1,\ldots ,p$\ and $%
X_{j}\sim G(\lambda t)$ for $j=p+1,\ldots ,n$, with $\lambda _{1},\lambda >0$
and $F$ is an absolutely continuous distribution. Let $X_{1}^{\ast },\ldots
,X_{n}^{\ast }$ be $n$
independent nonnegative random variables with $%
X_{i}^{\ast }\sim F(\lambda _{1}^{\ast }t)$ for $i=1,\ldots ,p$ and $%
X_{j}^{\ast }\sim G(\lambda t)$ for $j=p+1,\ldots ,n$, with $\lambda
_{1}^{\ast },\lambda >0$. Let $r_{F}$ and $r_{G}$ be the reverse hazard rate
functions of $F$ and $G$, respectively. Assume $tr_{F}(t)$ and $%
tr_{F}^{\prime }(t)/r_{F}(t)$ are both decreasing in $t$. Suppose $%
r_{F}(t)/r_{G}(t)$ is increasing in $t$, then%
\begin{equation*}
\lambda _{1}^{\ast }=\min (\lambda ,\lambda _{1},\lambda _{1}^{\ast
})\Rightarrow X_{n:n}\leq _{lr}X_{n:n}^{\ast }.
\end{equation*}
\end{theorem}

\begin{proof}
From (\ref{ecrh}) we get the reverse hazard rate function of $X_{n:n}$:%
\begin{equation*}
r_{n:n}(t)=p\lambda _{1}r_{F}(\lambda _{1}t)+q\lambda r_{G}(\lambda t),
\end{equation*}%
where $p+q=n$. Similarly the reverse hazard rate function of $X^{*}_{n:n}$
is
\begin{equation*}
r^{*}_{n:n}(t)=p\lambda^{*}_{1}r_{F}(\lambda _{1}t)+q\lambda r_{G}(\lambda
t).
\end{equation*}%
Observe that $X_{j} \overset{st}{=}  X_{j}^{\ast }$ for $j=p+1,\ldots ,n$. By the
assumptions, we know that $tr_{F}(t)$ is decreasing in $t$ and $\lambda
_{1}^{\ast }\leq \lambda _{1}$, then we have $X_{n:n}\leq _{rh}X_{n:n}^{\ast
}$ since $r_{n:n}(t)\leq r_{n:n}^{\ast }(t)$ for all $t$. From Theorem 1.C.4
in \cite{Shaked:2007}, it is enough to prove that the ratio of their reverse
hazard rate functions is increasing, i.e., we need to show that the function
\begin{equation*}
\phi (t)=\frac{r_{n:n}^{\ast }(t)}{r_{n:n}(t)}=\frac{p\lambda _{1}^{\ast
}r_{F}(\lambda _{1}^{\ast }t)+q\lambda r_{G}(\lambda t)}{p\lambda
_{1}r_{F}(\lambda _{1}t)+q\lambda r_{G}(\lambda t)}
\end{equation*}%
is increasing in $t$. On differentiating $\phi (t)$ with respect to $t$, we
get%
\begin{eqnarray*}
\phi ^{\prime }(t)  &\overset{\text{sign}}{=}&
t^{3}\left( p\left( \lambda _{1}^{\ast
}\right) ^{2}r_{F}^{\prime }(\lambda _{1}^{\ast }t)+q\lambda
^{2}r_{G}^{\prime }(\lambda t)\right) \left( p\lambda _{1}r_{F}(\lambda
_{1}t)+q\lambda r_{G}(\lambda t)\right) \\
&&-t^{3}\left( p\lambda _{1}^{\ast }r_{F}(\lambda _{1}^{\ast }t)+q\lambda
r_{G}(\lambda t)\right) \left( p\lambda _{1}^{2}r_{F}^{\prime }(\lambda
_{1}t)+q\lambda ^{2}r_{G}^{\prime }(\lambda t)\right) .
\end{eqnarray*}%
Let us denote:
\begin{equation*}
\psi _{F}\left( t\right) =tr_{F}(t)\text{, }\eta _{F}\left( t\right) =-t%
\frac{r_{F}^{\prime }(t)}{r_{F}(t)}\text{, }\psi _{G}\left( t\right)
=tr_{G}(t)\text{ and }\eta _{G}\left( t\right) =-t\frac{r_{G}^{\prime }(t)}{%
r_{G}(t)}\text{,}
\end{equation*}%
then the derivative of $\phi (t)$ can be rewritten as
\begin{eqnarray*}
\phi ^{\prime }(t) &\overset{\text{sign}}{=} &
p^{2}\psi _{F}(\lambda _{1}^{\ast
}t)\psi _{F}(\lambda _{1}t)\left( -\eta _{F}(\lambda _{1}^{\ast }t)+\eta
_{F}(\lambda _{1}t)\right) +pq\psi _{G}(\lambda t)\psi _{F}(\lambda
_{1}t)\left( -\eta _{G}(\lambda t)+\eta _{F}(\lambda _{1}t)\right) \\
&&+pq\psi _{F}(\lambda _{1}^{\ast }t)\psi _{G}(\lambda t)\left( -\eta
_{F}(\lambda _{1}^{\ast }t)+\eta _{G}(\lambda t)\right) \text{.}
\end{eqnarray*}%
The assumption $r_{F}(t)/r_{G}(t)$ is increasing in $t$ is equivalent to $%
\eta _{F}\left( t\right) \leq \eta _{G}\left( t\right) $ for all $t$. In
addition, we know that $\eta _{F}\left( t\right) $ is increasing in $t$ and $%
\lambda _{1}^{\ast }\leq \lambda $ then $\eta _{F}(\lambda _{1}^{\ast
}t)\leq \eta _{F}(\lambda t)\leq \eta _{G}(\lambda t)$ for all $t$. By the
assumptions, we know that $\psi _{F}\left( t\right) $ is decreasing in $t$
and $\lambda _{1}^{\ast }\leq \lambda _{1}$, then%
\begin{eqnarray*}
\phi ^{\prime }(t) &\geq &p^{2}\psi _{F}(\lambda _{1}^{\ast }t)\psi
_{F}(\lambda _{1}t)\left( -\eta _{F}(\lambda _{1}^{\ast }t)+\eta
_{F}(\lambda _{1}t)\right) +pq\psi _{G}(\lambda t)\psi _{F}(\lambda
_{1}t)\left( -\eta _{F}(\lambda _{1}^{\ast }t)+\eta _{F}(\lambda
_{1}t)\right) \\
&=&p\psi _{F}(\lambda _{1}t)\left( -\eta _{F}(\lambda _{1}^{\ast }t)+\eta
_{F}(\lambda _{1}t)\right) \left( p\psi _{F}(\lambda _{1}^{\ast }t)+q\psi
_{G}(\lambda t)\right) \geq 0,
\end{eqnarray*}%
since $\eta _{F}\left( t\right) $ is increasing in $t$. Therefore $%
r_{n:n}^{\ast }(t)/r_{n:n}(t)$ is increasing in $t$.
\end{proof}


\section*{Acknowledgements}
{%
This article is dedicated to our friend Ramesh Gupta for his many
contributions and for his encouragement to new researchers.
The authors wish to thank the Associate Editor and an
anonymous referee for their comments which have greatly
improved the initial version of this manuscript.
} 
The research of {Nuria Torrado}
was supported by the Portuguese Government through the
Funda\c{c}\~{a}o para a Ci\^{e}ncia e Tecnologia (FCT) under the grant
SFRH/BPD/91832/2012 and partially supported by the Centro de Matem\'{a}tica
da Universidade de Coimbra (CMUC) under the project PEst-C/MAT/UI0324/2013.


\end{document}